\documentclass{article}
\usepackage{amsmath,amssymb,bbm,theorem}

\newcommand{\assign}{:=}

\numberwithin{equation}{section} 

\newcommand{\nocomma}{}
\newcommand{\noplus}{}

\newcommand{\tmop}[1]{\ensuremath{\operatorname{#1}}}

\newcommand{\tmstrong}[1]{\textbf{#1}}
\newcommand{\tmtextbf}[1]{{\bfseries{#1}}}
\newcommand{\tmtexttt}[1]{{\ttfamily{#1}}}
\newenvironment{proof}{\noindent\textbf{Proof\ }}{\hspace*{\fill}$\Box$\medskip}
\newtheorem{corollary}{Corollary}

\newtheorem{lemma}{Lemma}
\newtheorem{proposition}{Proposition}
{\theorembodyfont{\rmfamily}}
\newtheorem{theorem}{Theorem}
\newtheorem{Conjecture}{Conjecture}

\newcommand{\XXint}[3]{{\setbox}0=\text{\ensuremath{#1 #2 #3 \int}}
{\vcenter{\text{\ensuremath{#2 #3}}}}{\kern}-.5{\tmwd}0}

\newcommand{\opn}[2]{\newcommand{\1}{\}} {\opn}{\Rm{Rm}} {\opn}{\Ric{Ric}}
{\opn}{\Rc{Rc}} {\opn}{\Scal{Sc}} {\opn}{\Tr{Tr}} {\opn}{\Trac{Tr}}
{\opn}detdet {\opn}{\diam{diam}} {\opn}{\dist{dist}} {\opn}{\Im}Im
{\opn}{\div}div {\opn}{\Ker{Ker}} {\opn}expexp {\opn}{\Vol{Vol}}
{\opn}{\exph{exph}} {\opn}{\Herm{Herm}} {\opn}{\End{End}} {\opn}{\Hess{Hess}}
{\opn}{\Vol{Vol}}}

\newcommand{\contract}{{\kern}-1.5pt{\vrule} width6.0pt height0.4pt depth0pt
{\vrule} width0.4pt height4.0pt depth0pt}
\newcommand{\retract}{{\kern}-1.5pt{\vrule} width0.4pt height4.0pt depth0pt
{\vrule} width6.0pt height0.4pt depth0pt}
\newcommand{\Openbox}{{\leavevmode} {\hfil}{\vrule} width{\boxrulethickness}
{\vbox} to{\Openboxwidth{{\advance}{\Openboxwidth} -2{\boxrulethickness}
{\hrule} height {\boxrulethickness} width{\Openboxwidth}{\vfil} {\hrule}
height{\boxrulethickness}}}{\vrule} width{\boxrulethickness}{\hfil} }

\begin{document}

\title{Chern-Ricci invariance along $G$-geodesics}\author{\\
{{NEFTON PALI}}}\date{}\maketitle

\begin{abstract}
  Over a compact oriented manifold, the space of Riemannian metrics and
  normalized positive volume forms admits a natural pseudo-Riemannian metric
  $G$, which is useful for the study of Perelman's $\mathcal{W}$ functional.
  We show that if the initial speed of a $G$-geodesic is $G$-orthogonal to the
  tangent space to the orbit of the initial point, under the action of the
  diffeomorphism group, then this property is preserved along all points of
  the $G$-geodesic. We show also that this property implies preservation of
  the Chern-Ricci form along such $G$-geodesics, under the extra assumption of
  complex anti-invariant initial metric variation and vanishing of the Nijenhuis tensor along the $G$-geodesic. 
  More in general we show that the main obstruction to the invariance of the Chern-Ricci form is the vanishing of the Nijenhuis tensor.
  This result is useful for
  a slice type theorem needed for the proof of the dynamical stability of the
  Soliton-K\"ahler-Ricci flow.
\end{abstract}
{\def\thefootnote{\relax}} 
\section{Statement of the invariance result}

We consider the space \tmtexttt{$\mathcal{M}$} of smooth Riemannian metrics
over a compact oriented manifold $X$ of dimension $m$. We denote by $\mathcal{V}_1$ the space
of positive smooth volume forms with integral one. Notice that the tangent
space of $\mathcal{M} \times \mathcal{V}_1$ is 
$$
T_{\mathcal{M} \times
\mathcal{V}_1} = C^{\infty} (X, S^2 T^{\ast}_X) \oplus C^{\infty} (X,
\Lambda^m T^{\ast}_X)_0,
$$
where
 $ C^{\infty} (X, \Lambda^m T^{\ast}_X)_0  \assign  \left\{ V \in C^{\infty}
  (X, \Lambda^m T^{\ast}_X) \mid \int_X V = 0 \right\} $.
We denote by $\tmop{End}_g \left( T_X \right)$ the bundle of $g$-symmetric
endomorphisms of $T_X$ and by $C_{\Omega}^{\infty} (X, \mathbbm{R})_0$ the
space of smooth functions with zero integral with respect to $\Omega$.
We will use the fact that for any $\left( g, \Omega \right) \in \mathcal{M}
\times \mathcal{V}_1$ the tangent space $T_{\mathcal{M} \times \mathcal{V}_1,
\left( g, \Omega \right)}$ identifies with $C^{\infty} (X, \tmop{End}_g \left(
T_X \right)) \oplus C_{\Omega}^{\infty} (X, \mathbbm{R})_0$ via the
isomorphism
\begin{eqnarray*}
  \left( v, V \right) & \longmapsto & \left( v^{\ast}_g, V_{\Omega}^{\ast}
  \right) \assign \left( g^{- 1} v, V / \Omega \right) .
\end{eqnarray*}
In \cite{Pal6}, we consider the pseudo-Riemannian metric $G$ over $\mathcal{M} \times
\mathcal{V}_1$, defined over any point $\left( g, \Omega \right) \in
\mathcal{M} \times \mathcal{V}_1$ by the formula
\begin{eqnarray*}
  G_{g, \Omega}  (u, U ; v, V) & = & \int_X \left[ \left\langle u, v
  \right\rangle_g - 2 U^{\ast}_{\Omega} V^{\ast}_{\Omega} \right] \Omega,
\end{eqnarray*}
for all $(u, U), (v, V) \in T_{\mathcal{M} \times \mathcal{V}_1}$. The
gradient flow of Perelman's $\mathcal{W}$-functional \cite{Per} with respect to the
structure $G$ is a modification of the Ricci flow with relevant properties
(see \cite{Pal6, Pal7}). The $G$-geodesics exists only for short time intervals
$\left( - \varepsilon, \varepsilon \right)$. This is because the $G$-geodesics
are uniquely determined by the evolution of the volume forms and the latter
degenerate in finite time (see section \ref{Pur-evolV}). In \cite{Pal6}, we show that the
space $G$-orthogonal to the tangent of the orbit of a point $\left( g, \Omega
\right) \in \mathcal{M} \times \mathcal{V}_1$, under the action of the
identity component of the diffeomorphism group is
\begin{eqnarray*}
  \mathbbm{F}_{g, \Omega} & \assign & \left\{ (v, V) \in T_{\mathcal{M} \times
  \mathcal{V}_1} \mid \nabla_g^{\ast_{\Omega}} v_g^{\ast} + \nabla_g
  V^{\ast}_{\Omega} = 0 \right\} ,
\end{eqnarray*}
where $\nabla_g^{\ast_{\Omega}}$ denotes the adjoint of the Levi-Civita connection with respect to the volume form $\Omega$.
In this paper we show the following conservative property.

\begin{proposition}
  \label{F-conserv} Let $\left( g_t, \Omega_t \right)_{t \in \left( -
  \varepsilon, \varepsilon \right)} \subset \mathcal{M} \times \mathcal{V}_1$
  be a $G$-geodesic such that \\$( \dot{g}_0, \dot{\Omega}_0) \in
  \mathbbm{F}_{g_0, \Omega_0}$. Then $( \dot{g}_t, \dot{\Omega}_t)
  \in \mathbbm{F}_{g_t, \Omega_t}$ for all $t \in \left( - \varepsilon,
  \varepsilon \right)$.
\end{proposition}

We consider now a compact symplectic manifold $\left( X, \omega \right)$ and
we denote by $\mathcal{J}_{\omega}^{\tmop{ac}}$ the space of smooth almost
complex structures compatible with the symplectic form $\omega$. We notice
that the variations inside the space of metrics
$\mathcal{M}^{\tmop{ac}}_{\omega} \assign - \omega \cdot
\mathcal{J}^{\tmop{ac}}_{\omega} \subset \mathcal{M}$, at a point $g = -
\omega J$, are $J$-anti-invariant. Thus, in this set-up, it is natural to
consider the sub-space
\begin{eqnarray*}
  \mathbbm{F}^J_{g, \Omega} & \assign & \left\{ \left( v, V \right) \in
  \mathbbm{F}_{g, \Omega} \mid v = - J^{\ast} v J_{_{_{_{}}}} \right\} .
\end{eqnarray*}
With these notations we state the following result.

\begin{theorem}
  \label{Main}{\tmstrong{$($Main result. The invariance of the Chern-Ricci
  form$)$.}}
  Let $\left( X, J_0, g_0 \right)$ be a compact almost-K\"ahler manifold with
  symplectic form \\$\omega \assign g_0 J_0$. Then for any
  $G$-geodesic $( g_t, \Omega_t)_{t \in \left( - \varepsilon,
  \varepsilon \right)} \subset \mathcal{M} \times \mathcal{V}_1$, with initial
  speed $( \dot{g}_0, \dot{\Omega}_0) \in \mathbbm{F}^{J_0}_{g_0,
  \Omega_0}$ holds the properties $J_t \assign - \omega^{- 1} g_t \in
  \mathcal{J}^{\tmop{ac}}_{\omega}$, $(\dot{g}_t, \dot{\Omega}_t)
  \in \mathbbm{F}^{J_t}_{g_t, \Omega_t}$ and the variation formulas
 \begin{eqnarray*}
    \frac{d}{d t} \tmop{Ric}_{J_t} (\Omega_t) & = & - \, d \tmop{Tr}_{g_t}\left[\omega(\bullet
    \neg \,N_{_{J_t}})\dot{g}_t^{\ast}\right] ,
  \end{eqnarray*}
 \begin{eqnarray*}
    2 \frac{d}{d t} \tmop{Ric}_{J_t} (\Omega_t) & = & d \tmop{Tr}_{g_t}\left[\omega(\bullet
    \neg \,\overline{\partial}_{T_{X, J_t}} \dot{g}_t^{\ast})\right] .
  \end{eqnarray*}  
In particular if the Nijenhuis tensor vanishes identically along the $G$-geodesic, then  
  $\tmop{Ric}_{_{J_t}}
  \left( \Omega_t \right)= \tmop{Ric}_{_{J_0}} \left( \Omega_0 \right)$, for all $t \in \left( - \varepsilon, \varepsilon
  \right)$.
\end{theorem}

Our unique interest in this result concerns the Fano case $\omega =\tmop{Ric}_{_{J_0}} \left( \Omega_0 \right)$. In
this case the space of $\omega$-compatible complex (integrable) structures
$\mathcal{J}_{\omega}$ embeds naturally inside $\mathcal{M} \times
\mathcal{V}_1$ via the Chern-Ricci form. (This is possible thanks to the
$\partial \bar{\partial}$-lemma). The image of this embedding is
\begin{eqnarray*}
  \mathcal{S}_{\omega} & \assign & \left\{ (g, \Omega) \in
  \mathcal{M}_{\omega} \times \mathcal{V}_1 \mid \omega = \tmop{Ric}_J
  (\Omega), J = - \omega^{- 1} g \right\},
\end{eqnarray*}
with $\mathcal{M}_{\omega} \assign - \omega \cdot \mathcal{J}_{\omega} \subset
\mathcal{M}$. It is well-known that the $J$-anti-linear endomorphism sections
associated to the metric-variations in $\mathcal{M}_{\omega}$ at a point $g =
- \omega J$, are $\overline{\partial}_{T_{X, J}}^{}$-closed. Thus, in the
integrable set-up, it is natural to consider the sub-space
\begin{eqnarray*}
  \mathbbm{F}^J_{g, \Omega} [0] & \assign & \left\{ \left( v, V \right) \in
  \mathbbm{F}^J_{g, \Omega} \mid \overline{\partial}_{T_{X, J}}^{} v_g^{\ast}
  = 0_{_{_{_{}}}} \right\} .
\end{eqnarray*}
It has been showed in \cite{Pal6} that this is the space $G$-orthogonal to the
tangent to the orbit of the point $\left( g, \Omega \right) \in
\mathcal{S}_{\omega}$, under the action of the identity component of the
$\omega$-symplectomorphisms group. (See the identity 1.14 in \cite{Pal6}). Furthermore
the product $G_{g, \Omega}$ is positive over $\mathbbm{F}^J_{g, \Omega} [0]$,
thanks to a result in \cite{Pal6}. We conjecture the following slice type result. 

\begin{Conjecture}
  Let $\left( X, J \right)$ be a Fano manifold and let $\omega \in 2 \pi c_1 \left( X \right)$ be a K\"ahler
  form. Then the distribution $\left( g, \Omega \right) \in
  \mathcal{S}_{\omega} \longmapsto \mathbbm{F}^J_{g, \Omega} \left[ 0
  \right]$, with $J \assign - \omega^{- 1} g$ is integrable over the space
  $\mathcal{S}_{\omega}$, with leave at the
  point $\left( g, \Omega \right)$ given locally, in a neighborhood of this point, by $\Sigma^{\omega}_{g, \Omega}
  \assign \{ \left( \gamma, \mu \right) \in \tmop{Exp}_G \left(
  \mathbbm{F}_{g, \Omega}^J \right) \mid \nabla_{\gamma} \omega = 0\}$. 
\end{Conjecture}
We invite the readers to compare with \cite{F-S} for other approaches concerning slice type problems in the space of compatible complex structures.
In view of the results in section 9 of \cite{Pal6}, the solution of this
conjecture is crucial for the proof of the dynamical stability of the
Soliton-K\"ahler-Ricci flow \cite{Pal9}. 
\\
An important ingredient for the proof of the main theorem \ref{Main} is the general variation formula \ref{Acvr-O-Rc-fm}. 
Particular cases of this variation formula have been intensively studied. See \cite{Fu, Do, Mo, Ga, Pal5}. 
These formulas allow to establish an 
important moment map picture in K\"{a}hler-geometry. See \cite{Fu} for the integrable case and \cite{Do} for the almost complex case.
In the last section we provide a formula relating the Bakry-Emery-Ricci tensor with the Chern-Ricci form.

\section{Pure evolving volume nature of the $G$-geodesic equation}\label{Pur-evolV}

We remind that the equation of a $G$-geodesic $\left( g_t, \Omega_t \right)_{t
\in \left( - \varepsilon, \varepsilon \right)}$, (see \cite{Pal6}), rewrites under the
form
\begin{eqnarray*}
\left( S \right)  \left\{ \begin{array}{l}
     \frac{d}{d t}  \dot{g}_t^{\ast} + \dot{\Omega}^{\ast}_t  \dot{g}_t^{\ast}
      = 0 ,\\
     \\
     \ddot{\Omega}_t + \frac{1}{4}  \left\{ \left| \dot{g}_t \right|^2_{g_t} -
     2 ( \dot{\Omega}^{\ast}_t)^2 - \int_X \left[ \left| \dot{g}_t
     \right|^2_{g_t} - 2 ( \dot{\Omega}_t^{\ast})^2 \right] \Omega_t \right\}
     \Omega_t   =  0  .
   \end{array} \right. 
\end{eqnarray*}   
The invariance of the scalar product of the speed of geodesics implies
\begin{eqnarray*}
  G_t \assign G_{g_t, \Omega_t} ( \dot{g}_t, \dot{\Omega}_t ; \dot{g}_t,
  \dot{\Omega}_t )  \equiv  G_{g_0, \Omega_0} ( \dot{g}_0,
  \dot{\Omega}_0 ; \dot{g}_0, \dot{\Omega}_0 ) .
\end{eqnarray*}
Therefore a solution of the system ($S$) satisfies also
\[ \left( S_1 \right)  \left\{ \begin{array}{l}
     \frac{d}{d t}  \dot{g}_t^{\ast} + \dot{\Omega}^{\ast}_t  \dot{g}_t^{\ast}
     = 0 ,\\
     \\
     \ddot{\Omega}_t + \frac{1}{4}  \left[ \left| \dot{g}_t \right|^2_{g_t} -
     2 ( \dot{\Omega}^{\ast}_t)^2 - G_0 \right] \Omega_t  =
      0  .
   \end{array} \right. \]
The first equation in the system ($S_1$) rewrites as
\begin{eqnarray*}
  \dot{g}_t^{\ast} & = & \frac{\Omega_0}{\Omega_t}  \dot{g}_0^{\ast},
\end{eqnarray*}
which provides the expression
\begin{eqnarray*}
  g_t & = & g_0 \exp \left( \dot{g}_0^{\ast} \int_0^t
  \frac{\Omega_0}{\Omega_s} d s \right) .
\end{eqnarray*}
We set $u_t \assign \Omega_t / \Omega_0$, and we observe the trivial
identities
\begin{eqnarray*}
  \left| \dot{g}_t \right|^2_{g_t} & = & \tmop{Tr}_{_{\mathbbm{R}}} \left(
  \dot{g}_t^{\ast} \right)^2 = u^{- 2}_t \left| \dot{g}_0 \right|^2_{g_0},\\
  &  & \\
  \dot{\Omega}_t^{\ast} & = & \dot{u}_t / u_t .
\end{eqnarray*}
We deduce that the system ($S_1$) is equivalent to the system
\[  \left\{ \begin{array}{l}
     g_t   = g_0 \exp \left( \dot{g}_0^{\ast} \int_0^t u^{-
     1}_s d s \right),\\
     \\
     \Omega_t   = u_t \Omega_0,\\
     \\
     4 \ddot{u}_t + \frac{\left| \dot{g}_0 \right|^2_{g_0} - 2
     \dot{u}_t^2}{u_t} - u_t G_0  =  0
     ,\\
     \\
     u_0   = 1,\\
     \\
     \int_X \dot{u}_0 \Omega_0 = 0 .
   \end{array} \right. \]
The solution $u$ is given by the explicit formula
\begin{eqnarray*}
  u_t & = & 1 + \dot{u}_0 \sum_{k \geqslant 0}  \frac{\left( G_0 / 2
  \right)^k}{\left( 2 k + 1 \right) !} t^{2 k + 1} - \frac{1}{4} 
  \underline{N_0}  \sum_{k \geqslant 1} \frac{\left( G_0 / 2 \right)^{k -
  1}}{\left( 2 k \right) !} t^{2 k},\\
  &  & \\
  \underline{N_0} & : = & N_0 - G_0,\\
  &  & \\
  N_0 & \assign & \left| \dot{g}_0 \right|^2_{g_0} - 2 (
  \dot{\Omega}^{\ast}_0)^2 .
\end{eqnarray*}
Thus the solution $\left( g_t, \Omega_t \right)_{t \in \left( - \varepsilon,
\varepsilon \right)}$ of the system ($S_1$) satisfies $\int_X \Omega_t \equiv
1$. This implies $G_t \equiv G_0$. We infer that the system ($S_1$) is
equivalent to the system ($S$).

In the case $G_0 > 0$, the previous formula for $u_t$ reduces to the
expression
\begin{eqnarray*}
  u_t  =  1 + \dot{\Omega}^{\ast}_0 \gamma^{- 1}_0 \sinh \left( \gamma_0 t
  \right) - \underline{N_0}  \left( 2 \gamma_0 \right)^{- 2}  [\cosh \left(
  \gamma_0 t \right) - 1],
\end{eqnarray*}
with $\gamma_0 : = \left( G_0 / 2 \right)^{1 / 2}$.

\section{Conservative properties along $G$-geodesics}

In this section we show proposition \ref{F-conserv}.

\begin{proof}
  We remind first the fundamental variation formula
  \begin{eqnarray}
    2 \left[ \left( D_{g, \Omega} \nabla_{\bullet}^{\ast_{\bullet}} \right)
    \left( v, V \right)_{_{_{_{}}}} \right] v^{\ast}_g  =  \frac{1}{2}
    \nabla_g |v|^2_g - 2 v^{\ast}_g \cdot \left( \nabla_g^{\ast_{\Omega}}
    v^{\ast}_g + \nabla_g V^{\ast}_{\Omega} \right), \label{super-var-Div} 
  \end{eqnarray}
  obtained in \cite{Pal8}, (see the formula 19 in \cite{Pal8}). Using (\ref{super-var-Div}) we develop the
  derivative
  \begin{eqnarray*}
    2 \frac{d}{d t}  \left( \nabla_{g_t}^{\ast_{\Omega_t}}  \dot{g}^{\ast}_t +
    \nabla_{g_t}  \dot{\Omega}^{\ast}_t \right) & = & - 2 \dot{g}^{\ast}_t
    \cdot \left( \nabla_{g_t}^{\ast_{\Omega_t}}  \dot{g}^{\ast}_t +
    \nabla_{g_t}  \dot{\Omega}^{\ast}_t \right) + \frac{1}{2} \nabla_{g_t} |
    \dot{g}_t |_{g_t}^2\\
    &  & \\
    & + & 2 \nabla_{g_t}^{\ast_{\Omega_t}} \frac{d}{d t}  \dot{g}_t^{\ast}
    \noplus + 2 \nabla_{g_t} \frac{d}{d t}  \dot{\Omega}_t^{\ast} - 2
    \dot{g}^{\ast}_t \cdot \nabla_{g_t}  \dot{\Omega}^{\ast}_t .
  \end{eqnarray*}
  Writing the equations defining the $G$-geodesic $\left( g_t, \Omega_t
  \right)_{t \in \left( - \varepsilon, \varepsilon \right)}$, under the form
  \[ \left\{ \begin{array}{l}
       \frac{d}{d t}  \dot{g}_t^{\ast} + \dot{\Omega}^{\ast}_t 
       \dot{g}_t^{\ast}   = 0 ,\\
       \\
       2 \frac{d}{d t}  \dot{\Omega}_t^{\ast} + ( \dot{\Omega}^{\ast}_t)^2 +
       \frac{1}{2}  \left| \dot{g}_t \right|^2_{g_t} - \frac{1}{2} G_{g_t,
       \Omega_t} ( \dot{g}_t, \dot{\Omega}_t ; \dot{g}_t, \dot{\Omega}_t)
        =  0,
     \end{array} \right. \]
  we infer
  \begin{eqnarray*}
    2 \frac{d}{d t}  \left( \nabla_{g_t}^{\ast_{\Omega_t}}  \dot{g}^{\ast}_t +
    \nabla_{g_t}  \dot{\Omega}^{\ast}_t \right) & = & - 2 \dot{g}^{\ast}_t
    \cdot \left( \nabla_{g_t}^{\ast_{\Omega_t}}  \dot{g}^{\ast}_t +
    \nabla_{g_t}  \dot{\Omega}^{\ast}_t \right)\\
    &  & \\
    & - & 2 \nabla_{g_t}^{\ast_{\Omega_t}} \left( \dot{\Omega}^{\ast}_t 
    \dot{g}_t^{\ast} \right) \noplus - \nabla_{g_t} \left(
    \dot{\Omega}^{\ast}_t \right)^2 - 2 \dot{g}^{\ast}_t \cdot \nabla_{g_t} 
    \dot{\Omega}^{\ast}_t,
  \end{eqnarray*}
  and thus
  \begin{eqnarray*}
    2 \frac{d}{d t}  \left( \nabla_{g_t}^{\ast_{\Omega_t}}  \dot{g}^{\ast}_t +
    \nabla_{g_t}  \dot{\Omega}^{\ast}_t \right)  =  - 2 \left(
    \dot{g}^{\ast}_t + \dot{\Omega}^{\ast}_t \mathbbm{I} \right) \cdot \left(
    \nabla_{g_t}^{\ast_{\Omega_t}}  \dot{g}^{\ast}_t + \nabla_{g_t} 
    \dot{\Omega}^{\ast}_t \right) .
  \end{eqnarray*}
  Then the conclusion follows by Cauchy's uniqueness.
\end{proof}

Let $\mathcal{J} \subset C^{\infty} (X, \tmop{End}_{\mathbbm{R}} (T_X))$ be
the set of smooth almost complex structures over $X$. For any non degenerate closed
$2$-form $\omega$ over a symplectic manifold, we define the space
$\mathcal{J}^{\tmop{ac}}_{\omega}$ of $\omega$-compatible almost complex
structures as
\begin{eqnarray*}
  \mathcal{J}^{\tmop{ac}}_{\omega} & \assign & \{J \in \mathcal{J} \mid - \;
  \omega J \in \mathcal{M}\} .
\end{eqnarray*}
With these notations holds the following result.

\begin{lemma}
  \label{acx-geodesic}Let $J_0 \in \mathcal{J}^{\tmop{ac}}_{\omega}$ and let
  $\left( g_t, \Omega_t \right)_{t \in \left( - \varepsilon, \varepsilon
  \right)} \subset \mathcal{M} \times \mathcal{V}_1$ be a $G$-geodesic such
  that $g_0 = - \omega J_0$ and $\dot{g}^{\ast}_0 J_0 = - J_0 
  \dot{g}^{\ast}_0$. Then $J_t \assign - \omega^{- 1} g_t \in
  \mathcal{J}^{\tmop{ac}}_{\omega}$, for all $t \in \left( - \varepsilon,
  \varepsilon \right)$.
\end{lemma}

\begin{proof}
  Using the identity $\dot{J}_t = J_t  \dot{g}_t^{\ast}$ and the $G$-geodesic
  equation
  \begin{eqnarray*}
    \frac{d}{d t}  \dot{g}_t^{\ast} + \dot{\Omega}^{\ast}_t  \dot{g}^{\ast}_t
    & = & 0,
  \end{eqnarray*}
  we obtain the variation formula
  \begin{eqnarray*}
    \frac{d}{d t}  \left( J_t  \dot{g}_t^{\ast} + \dot{g}_t^{\ast} J_t \right)
    & = & \left( J_t  \dot{g}_t^{\ast} + \dot{g}_t^{\ast} J_t \right) \cdot
    \left( \dot{g}^{\ast}_t - \dot{\Omega}^{\ast}_t \mathbbm{I} \right) .
  \end{eqnarray*}
  This implies $\dot{g}^{\ast}_t J_t = - J_t  \dot{g}^{\ast}_t$, for all $t
  \in \left( - \varepsilon,
  \varepsilon \right)$, by Cauchy's uniqueness. We deduce in particular the
  evolution identity $2 \dot{J}_t = \left[ J_t, \dot{g}_t^{\ast} \right]$.
  Then $J^2_t = -\mathbbm{I}_{_{T_X}}$, thanks to lemma 4 in \cite{Pal3}. We
  infer the required conclusion.
\end{proof}

\section{The first variation of the $\Omega$-Chern-Ricci form}

Let $(X, J)$ be an almost complex manifold. Any volume form $\Omega > 0$
induces a hermitian metric $h_{\Omega}$ over the canonical bundle $K_{X, J}
\assign \Lambda_J^{n, 0} T^{\ast}_X$, which is given by the formula
\[ h_{\Omega} (\alpha, \beta) \assign \frac{n! i^{n^2} \alpha \wedge
   \overline{\beta} }{\Omega} . \]
We define the $\Omega$-Chern-Ricci form
\[ \tmop{Ric}_{_J} \left( \Omega \right) \assign - i\mathcal{C}_{h_{\Omega}}
   \left( K_{X, J} \right), \]
where $\mathcal{C}_h (F)$ denotes the Chern curvature of a hermitian vector
bundle $(F, \overline{\partial}_F, h)$, equipped with a $\left( 0, 1
\right)$-type connection. Consider also a $J$-invariant hermitian metric
$\omega$ over $X$. We remind that the $\omega$-Chern-Ricci form is defined by
the formula
\begin{eqnarray*}
  \tmop{Ric}_{_J} \left( \omega \right)  \assign  \tmop{Tr}_{_{\mathbbm{C}}}
  \left[ J\mathcal{C}_{\omega} \left( T_{X, J} \right) \right] .
\end{eqnarray*}
The fact that the metric $h_{\omega^n}$ over $K_{X, J}$ is induced by the
metric $\omega$ over $T_{X, J}$ implies, by natural functorial properties, the
identity $\tmop{Ric}_{_J} (\omega) = \tmop{Ric}_{_J} (\omega^n)$. Let now
$$
\mathcal{KS}  : = \Big\{ (J, g) \in
   \mathcal{J} \times \mathcal{M}  \mid 
   g = J^{\ast} gJ
   , \; \nabla_g J 
   =  0 \Big\} , 
$$
be the space of K\"ahler structures over a compact manifold $X$. We remind
that if $A \in \tmop{End}_{\mathbbm{R}} (T_X)$, then its transposed $A^T_g$
with respect to $g$ is given by $A^T_g = g^{- 1} A^{\ast} g$. We observe that
the compatibility condition $g = J^{\ast} gJ$, is equivalent to the condition
$J^T_g = - J$. We define also the space of almost K\"ahler structures as
\[ \mathcal{A} \mathcal{KS} : =  \Big\{ (J, g)
   \in \mathcal{J} \times \mathcal{M}  \mid 
    g  =  J^{\ast} gJ
   ,  d \left( g J \right)
   =  0 \Big\}. \]
With these notations hold the following first variation formula for the
$\Omega$-Chern-Ricci form. (Compare with \cite{Fu, Do, Mo, Ga, Pal5}).

\begin{proposition}
  \label{AC-var-O-Rc-fm}Let $(J_t, g_t)_t \subset
  \mathcal{A}\mathcal{K}\mathcal{S}$ and $(\Omega_t)_t \subset \mathcal{V}$ be
  two smooth paths such that $\dot{J}_t = ( \dot{J}_t)_{g_t}^T$. Then hold the
  first variation formula
  \begin{equation}
    \label{Acvr-O-Rc-fm} 2 \frac{d}{d t} \tmop{Ric}_{_{J_t}} (\Omega_t) =
    L_{J_t\nabla_{g_t}^{\ast_{\Omega_t}}\dot{J}_t -
    \nabla_{g_t}  \dot{\Omega}^{\ast}_t} \omega_t,
  \end{equation}
  with $\omega_t = g_t J_t$.
\end{proposition}

\begin{proof}
  {\tmstrong{STEP I. Local expressions}}. We consider first the case of
  constant volume form $\Omega$. We remind a general basic identity. Let $(L,
  \overline{\partial}_L, h)$ be a hermitian line bundle, equipped with a
  $\left( 0, 1 \right)$-type connection, over an almost complex manifold $(X,
  J)$ and let $D_{L, h} = \partial_{L, h} + \overline{\partial}_L$ be the
  induced Chern connection. In explicit terms $\partial_{L, h} \assign h^{- 1}
  \cdot \partial_{\overline{L^{\ast}}} \cdot h$. We observe that for any local
  non-vanishing section $\sigma \in C^{\infty} (U, L \smallsetminus 0)$ over
  an open set $U \subset X$, holds the identity
  \begin{eqnarray*}
    \sigma^{- 1} \partial_{L, h} \sigma (\eta) & = & \left. | \sigma |_h^{- 2}
    h ( \partial_{L, h} \sigma (\eta), \sigma \right)\\
    &  & \\
    & = & | \sigma |_h^{- 2} \left[ \eta_J^{1, 0} . | \sigma |_h^2 - h (
    \sigma, \overline{\partial}_L \sigma ( \overline{\eta}) \right)]\\
    &  & \\
    & = & \eta_J^{1, 0} . \log | \sigma |_h^2 - \overline{\sigma^{- 1}
    \overline{\partial}_L \sigma ( \overline{\eta})},
  \end{eqnarray*}
  for all $\eta \in T_X \otimes_{_{\mathbbm{R}}} \mathbbm{C}$. We infer the
  formula
  \begin{eqnarray*}
    i \sigma^{- 1} D_{L, h} \sigma & = & \left. i \partial_{_J} \log | \sigma
    |_h^2 + 2 \Re e ( i \sigma^{- 1} \overline{\partial}_L \sigma \right) .
  \end{eqnarray*}
  In the case $L = K_{X, J_t} \assign \Lambda_{J_t}^{n, 0} T^{\ast}_X $ and $h
  \equiv h_{\Omega}$ we get for all
  \[ \beta_t = \beta^{1, 0}_{1, t} \wedge \ldots \wedge \beta^{1, 0}_{n, t}
     \in C^{\infty} \left( U, K_{X, J_t} \smallsetminus 0 \right), \]
 with $\beta^{1,
     0}_{r, t} \assign \beta^{1, 0}_{r, J_t}$, $\beta_r \in
     C^{\infty} \left( U, T^{\ast}_X \otimes_{_{\mathbbm{R}}} \mathbbm{C}
     \right)$,  $r = 1, \ldots, n$,
  the formula for the $1$-form $\alpha_t $,
  \begin{eqnarray*}
    \alpha_t \assign i \beta_t^{- 1} D_{K_{X, J_t}, h_{\Omega}} \beta_t  = 
    i \partial_{_{J_t}} \log \frac{i^{n^2} \beta_t \wedge
    \bar{\beta}_t}{\Omega} + 2 \Re e \left( i \beta_t^{- 1}
    \overline{\partial}_{K_{X, J_t}} \beta_t \right) .
  \end{eqnarray*}
  We also notice the local expression $\tmop{Ric}_{J_t} (\Omega) = -
  i\mathcal{C}_{h_{\Omega}} \left( K_{X, J_t} \right) = - d \alpha_t$. In
  order to expand the time derivative of the expression
  \begin{eqnarray*}
    \alpha_t \left( \eta \right) & = & i \eta_{_{J_t}}^{1, 0} . \log
    \frac{i^{n^2} \beta_t \wedge \bar{\beta}_t}{\Omega}  \\
    & &\\
    &+& 2 \Re e
    \left[ i \beta_t^{- 1} \sum_{r = 1}^n \beta^{1, 0}_{1, t} \wedge \ldots
    \wedge \left( \eta_{_{J_t}}^{0, 1} \neg \overline{\partial}_{_{J_t}}
    \beta^{1, 0}_{r, t} \right) \wedge \ldots \wedge \beta^{1, 0}_{n, t}
    \right],
  \end{eqnarray*}
  we observe first the formula
  \begin{equation}
    \label{varDBar} 2 \frac{d}{d t}  \left( \overline{\partial}_{_{J_t}}
    \beta^{1, 0}_{_{J_t}} \right) = J_t  \dot{J}_t \neg \left[ \left( d - 2
    \overline{\partial}_{_{J_t}} \right) \beta^{1, 0}_{_{J_t}} \right] - i
    \left[ d \left( \beta \cdot \dot{J}_t \right) \right]_{_{J_t}}^{1, 1}.
  \end{equation}
  We notice indeed that for bi-degree reasons holds the identity
  \begin{eqnarray*}
    2 \overline{\partial}_{_{J_t}} \beta^{1, 0}_{_{J_t}} & = & 2 \left( d
    \beta^{1, 0}_{_{J_t}} \right)_{J_t}^{1, 1} = \; d \beta^{1, 0}_{_{J_t}} +
    J^{\ast}_t d \beta^{1, 0}_{_{J_t}} J_t .
  \end{eqnarray*}
  Then time deriving the latter we infer the required formula (\ref{varDBar}).
  
  {\tmstrong{STEP II. Local choices}}. We fix an arbitrary time $\tau$. We
  want to compute the time derivative $\dot{\alpha}_{\tau} (\eta)$. We take
  the open set $U \subset X$ relatively compact. Then for a sufficiently small
  $\varepsilon > 0$, the bundle map
  \begin{eqnarray*}
    \varphi_t \assign \det_{_{\mathbbm{C}}} \pi_{_{J_t}}^{1, 0} :
    \Lambda_{J_{\tau}}^{n, 0} T^{\ast}_U & \longrightarrow & \Lambda_{J_t}^{n,
    0} T^{\ast}_U \\
    &  & \\
    \beta_1 \wedge \ldots \wedge \beta_n & \longmapsto & \beta_t \assign
    \beta^{1, 0}_{1, t} \wedge \ldots \wedge \beta^{1, 0}_{n, t}  \;,
  \end{eqnarray*}
  is an isomorphism for all $t \in \left( \tau - \varepsilon, \tau +
  \varepsilon \right)$. We set for notations simplicity $D_t \assign D_{K_{X, J_t},
  h_{\Omega}}$. We consider also the connection $D_{\varphi_t} \assign
  \varphi^{\ast}_t D_t$ over the bundle $\Lambda_{J_{\tau}}^{n, 0}
  T^{\ast}_U$. Explicitly $D_{\varphi_t} \beta = \varphi_t^{- 1} D_t
  \beta_t$. Then the expression $D_t \beta_t = \alpha_t \otimes \beta_t$,
  implies $D_{\varphi_t} \beta = \alpha_t \otimes \beta$. We deduce that
  \begin{eqnarray*}
    \dot{\alpha}_t & = & \frac{d}{d t} D_{\varphi_t},
  \end{eqnarray*}
  is independent of the choice of $\beta$. We want to compute
  $\dot{\alpha}_{\tau}$ at an arbitrary point $p \in U$. 
  
  {\tmstrong{STEP IIa. The K\"ahler case}}. (We consider first this case
  since is drastically simpler). Let $\nabla_{g_{\tau}}$ be the Levi-Civita
  connection of $g_{\tau}$. Using parallel transport and the K\"ahler
  assumption $\nabla_{g_{\tau}} J_{\tau} = 0$, we can construct (up to
  shrinking $U$ around $p$), a frame $( \beta_r )^n_{r = 1}
  \subset C^{\infty} ( U, \Lambda_{J_{\tau}}^{1, 0} T^{\ast}_U)$,
  satisfying $\nabla_{g_{\tau}} \beta_r  \left( p \right) = 0$, for all $r =
  1, \ldots, n$, and the identity 
  $$
  \omega_{\tau} = \frac{i}{2}  \sum_{r = 1}^n
  \beta_r \wedge \bar{\beta}_r ,
  $$ over $U$. Then $d V_{g_{\tau}} = 2^{- n}
  i^{n^2} \beta_{\tau} \wedge \bar{\beta}_{\tau}$. We set now $f_{\tau}
  \assign \log \frac{d V_{g_{\tau}}}{\Omega}$. The identity $d \beta_r =
  \tmop{Alt} \nabla_{g_{\tau}} \beta_r$, implies $d \beta_r \left( p \right) =
  0$. Then formula (\ref{varDBar}) implies the identity at the point $p$,
  \begin{eqnarray*}
    2 \frac{d}{d t} _{\mid_{t = \tau}} \left( \overline{\partial}_{_{J_t}}
    \beta^{1, 0}_{r, t} \right) & = & - i \beta_r  \left( \nabla_{_{T_X,
    g_{\tau}}}  \dot{J}_{\tau} \right)_{_{J_{\tau}}}^{1, 1}\\
    &  & \\
    & = & - i \beta_r \tmop{Alt} \left( \nabla^{1, 0}_{g_{\tau}, J_{\tau}} 
    \dot{J}_{\tau} \right) .
  \end{eqnarray*}
  (The last equality follows from the K\"ahler assumption). We deduce
  \begin{eqnarray*}
    \eta_{_{J_{\tau}}}^{0, 1} \neg 2 \frac{d}{d t} _{\mid_{t = \tau}} \left(
    \overline{\partial}_{_{J_t}} \beta^{1, 0}_{r, t} \right) & = & i \beta_r
    \nabla^{1, 0}_{g_{\tau}, J_{\tau}}  \dot{J}_{\tau} \cdot
    \eta_{_{J_{\tau}}}^{0, 1} \; = \; i \beta_r \nabla^{1, 0}_{g_{\tau},
    J_{\tau}}  \dot{J}_{\tau} \cdot \eta,
  \end{eqnarray*}
  at the point $p$. (The last equality follows also from the K\"ahler
  assumption). Using this last identity we obtain the expression at the point
  $p$,
  \begin{eqnarray*}
    \dot{\alpha}_{\tau} (\eta) & = & \frac{1}{2}  \dot{J}_{\tau} \eta .
    f_{\tau} + \eta_{_{J_{\tau}}}^{1, 0} . \Re e \left( \beta^{- 1}_{\tau}
    \sum_{r = 1}^n \beta_1 \wedge \ldots \wedge \left( \beta_r  \dot{J}_{\tau}
    \right)_{_{J_{\tau}}}^{1, 0} \wedge \ldots \wedge \beta_n \right)\\
    &  & \\
    & - & \Re e \left( \beta^{- 1}_{\tau} \sum_{r = 1}^n \beta_1 \wedge
    \ldots \wedge \beta_r \nabla^{1, 0}_{g_{\tau}, J_{\tau}}  \dot{J}_{\tau}
    \cdot \eta \wedge \ldots \wedge \beta_n \right)\\
    &  & \\
    & = & - \frac{1}{2}  \left( \tmop{Tr}_{_{\mathbbm{R}}} \nabla_{g_{\tau}} 
    \dot{J}_{\tau} - d f_{\tau} \cdot \dot{J}_{\tau} \right) (\eta),
  \end{eqnarray*}
  thanks to the elementary identities $( \beta_r  \dot{J}_{\tau})_{_{J_{\tau}}}^{1, 0} = 0$, $\tmop{Tr}_{_{\mathbbm{C}}} A =
  \tmop{Tr}_{_{\mathbbm{C}}} A^{\ast}$, and \\$\tmop{Tr}_{_{\mathbbm{R}}} B = 2
  \Re e \left( \tmop{Tr}_{_{\mathbbm{C}}} B_{_J}^{1, 0} \right)$, for all $B
  \in \tmop{End}_{_{\mathbbm{R}}} (T_X)$. Using now the symmetry identities
  $\dot{J}_{\tau} = ( \dot{J}_{\tau} )_{g_{\tau}}^T$ and
  $\nabla_{g_{\tau}, \xi}  \dot{J}_{\tau} =( \nabla_{g_{\tau}, \xi} 
  \dot{J}_{\tau} )_{g_{\tau}}^T$, we obtain
  \begin{eqnarray*}
    2 \dot{\alpha}_{\tau} & = & \nabla_{g_{\tau}}^{\ast_{\Omega}} 
    \dot{J}_{\tau} \neg g_{\tau},
  \end{eqnarray*}
  over $U$. We conclude, thanks to the K\"ahler condition and Cartan's
  identity, the required formula for arbitrary time $t$, in the case of
  constant volume form.
  
  {\tmstrong{STEP IIb. The almost K\"ahler case}}. We remind first that in
  this case holds the classical identity
  \begin{equation}
    \label{covJ-Nij} g \left( \nabla_{g, \xi} J \cdot \eta, \mu \right) = - 2
    g \left( J \xi, N_{_J} \left( \eta, \mu \right) \right),
  \end{equation}
  where $N_{_J}$ is the Nijenhuis tensor, defined by the formula
  \begin{eqnarray*}
    4 N_{_J} \left( \xi, \eta \right) & \assign & [\xi, \eta] + J [\xi, J
    \eta] + J [J \xi, \eta] - [J \xi, J \eta] .
  \end{eqnarray*}
  The identity (\ref{covJ-Nij}) combined with the identity $N_{_J} \left( J \eta, \mu \right) = - J
  N_{_J} \left( \eta, \mu \right)$, implies
  \begin{equation}
    \label{anti-Lin-CovJ} \nabla_{g, J \xi} J = - J \nabla_{g, \xi} J .
  \end{equation}
  We consider also the Chern connection $D_{\Lambda_{_J}^{1, 0}
  T^{\ast}_X}^{\omega}$ of the complex vector bundle $\Lambda_J^{1, 0}
  T^{\ast}_X$ with respect to the hermitian product
  \begin{eqnarray*}
    \langle \alpha, \beta \rangle_{_{\omega}} & \assign & \frac{1}{2}
    \tmop{Tr}_{\omega} \left( i \alpha \wedge \bar{\beta} \right) .
  \end{eqnarray*}
  This connection is obviously the dual of the Chern connection $D_{T_{X,
  J}^{1, 0}}^{\omega_{\tau}}$ of the hermitian vector bundle $( T_{X,
  J}^{1, 0}, \omega)$. We denote by $D_{T_{X, J}}^{\omega}$ the Chern
  connection of $\left( T_{X, J}, \omega \right)$. By abuse of notations, we
  denote with the same symbol its complex linear extension over $T_X
  \otimes_{_{\mathbbm{R}}} \mathbbm{C}$. The latter satisfies the formula
  \begin{eqnarray*}
    D_{T_{X, J}}^{\omega} \xi & = & D_{T_{X, J}^{1, 0}}^{\omega} \xi^{1,
    0}_{_J} + \overline{^{^{^{^{}}}} D_{T_{X, J}^{1, 0}}^{\omega} 
    \overline{\xi_{_J}^{0, 1}}}, \hspace{2em} \forall \xi \in C^{\infty}
    \left( X, T_X \otimes_{_{\mathbbm{R}}} \mathbbm{C} \right) .
  \end{eqnarray*}
  In the almost K\"ahler case, $D_{T_{X, J}}^{\omega}$ is related to the
  Levi-Civita connection $\nabla_g$ (see \cite{Pal2} and use
  identity (\ref{covJ-Nij})), via the formula
  \begin{equation}
    \label{Chern-LC} D_{T_{X, J}, \xi}^{\omega} \eta \; = \; \nabla_{g_{}
    \nocomma \nocomma, \xi} \eta - \frac{1}{2} J \nabla_{g_{} \nocomma
    \nocomma, \xi} J \cdot \eta \nocomma,
  \end{equation}
  for all $\xi, \eta \in C^{\infty} \left( X, T_X \otimes_{_{\mathbbm{R}}}
  \mathbbm{C} \right)$. Thus
  \begin{eqnarray*}
    D_{T_{X, J}^{1, 0}}^{\omega} \xi^{1, 0}_{_J} & = & \nabla_g \xi^{1,
    0}_{_J} - \frac{1}{2} J \nabla_{g_{} \nocomma \nocomma} J \cdot \xi^{1,
    0}_{_J},
  \end{eqnarray*}
  and
  \begin{eqnarray*}
    D^{\omega}_{\Lambda_{_J}^{1, 0} T^{\ast}_X} \beta_{_J}^{1, 0} \cdot
    \xi^{1, 0}_{_J} & = & \nabla_g \beta_{_J}^{1, 0} \cdot \xi^{1, 0}_{_J} +
    \frac{1}{2} i \beta_{_J}^{1, 0} \cdot \nabla_{g_{} \nocomma \nocomma} J
    \cdot \xi^{1, 0}_{_J} \\
    &  & \\
    & = & \nabla_g \beta_{_J}^{1, 0} \cdot \xi^{1, 0}_{_J},
  \end{eqnarray*}
  since $\nabla_{g_{} \nocomma \nocomma} J\cdot J = - J \nabla_{g_{} \nocomma
  \nocomma} J$. We apply now these considerations to the almost K\"ahler
  structure $\left( J_{\tau}, g_{\tau} \right)$. Using parallel transport, we
  can construct a complex frame $\left( \beta_r \right)^n_{r = 1} \subset
  C^{\infty} ( U, \Lambda_{J_{\tau}}^{1, 0} T^{\ast}_U)$,
  satisfying $D^{\omega_{\tau}}_{\Lambda_{_{J_{\tau}}}^{1, 0} T^{\ast}_X}
  \beta_r  \left( p \right) = 0$, for all $r = 1, \ldots, n$, and the identity
  $$
  \omega_{\tau} = \frac{i}{2}  \sum_{r = 1}^n \beta_r \wedge \bar{\beta}_r ,
  $$
  over $U$. Then as before, holds the identity $d V_{g_{\tau}} = 2^{- n} i^{n^2}
  \beta_{\tau} \wedge \bar{\beta}_{\tau}$. We infer
  \begin{equation}
    \label{vanishDbar} \overline{\partial}_{_{J_{\tau}}} \beta_r \left( p
    \right) = 0,
  \end{equation}
  \begin{equation}
    \label{vanishLVCivita}  \left( \nabla_{g_{\tau}} \beta_r \cdot \xi^{1,
    0}_{_{J_{\tau}}} \right) \left( p \right) = 0,
  \end{equation}
  \begin{equation}
    \label{vanishDel} \partial_{_{J_{\tau}}} \beta_r \left( p \right) = 0,
  \end{equation}
  for all $r = 1, \ldots, n$. The last one follows indeed from the elementary
  identities
  \begin{eqnarray*}
    \partial_{_{J_{\tau}}} \beta_r  \left( \xi, \eta \right)  \;= \; d \beta_r
    \left( \xi, \eta \right)  \;=\;  \nabla_{g_{\tau}, \xi} \beta_r \cdot \eta
    - \nabla_{g_{\tau}, \eta} \beta_r \cdot \xi,  
  \end{eqnarray*}
  for all $\xi,
    \eta \in C^{\infty}( X, T^{1, 0}_{X, J_{\tau}})$.
  We observe now that formula (\ref{varDBar}) writes as
  \begin{eqnarray*}
    2 \frac{d}{d t}  \left( \overline{\partial}_{_{J_t}} \beta^{1, 0}_{_{J_t}}
    \right)  =  J_t  \dot{J}_t \neg \left[ \left( \partial_{_{J_t}} -
    \overline{\partial}_{_{J_t}} \right) \beta^{1, 0}_{_{J_t}} - \beta^{1,
    0}_{_{J_t}} \cdot N_{_{J_t}} \right] - i \left[ d \left( \beta \cdot
    \dot{J}_t \right) \right]_{_{J_t}}^{1, 1} .
  \end{eqnarray*}
  Thus at the point $p$ holds the identity
\begin{eqnarray*}
2 \frac{d}{d t} _{\mid_{t = \tau}} \left( \overline{\partial}_{_{J_t}}
    \beta^{1, 0}_{r, t} \right) \left( \eta_{_{J_{\tau}}}^{0, 1}, \mu \right)
    & = & -\,\beta_r \cdot N_{_{J_{\tau}}} \left( \eta_{_{J_{\tau}}}^{0, 1},
    J_{\tau}  \dot{J}_{\tau} \mu \right) \\
    & &\\
    &-& i \left[ d \left( \beta_r \cdot
    \dot{J}_{\tau} \right) \right]_{_{J_{\tau}}}^{1, 1} \left(
    \eta_{_{J_{\tau}}}^{0, 1}, \mu \right) \\
    &  & \\
    & = & i\, \beta_r \cdot N_{_{J_{\tau}}} \left( \eta, \dot{J}_{\tau} \mu
    \right) \\
    & &\\
    &-& i \,d \left( \beta_r \cdot \dot{J}_{\tau} \right) \left(
    \eta_{_{J_{\tau}}}^{0, 1}, \mu_{_{J_{\tau}}}^{1, 0} \right) .
  \end{eqnarray*}
  We set for notations simplicity $\eta_{\tau}^{0, 1} \assign
  \eta_{_{J_{\tau}}}^{0, 1}$, $\mu_{\tau}^{1, 0} : = \mu_{_{J_{\tau}}}^{1, 0}$
  and we observe the expansion
  \begin{eqnarray*}
    d \left( \beta_r \cdot \dot{J}_{\tau} \right) \left( \eta_{\tau}^{0, 1},
    \mu_{\tau}^{1, 0} \right) & = & \nabla_{g_{\tau}, \eta_{\tau}^{0, 1}}
    \beta_r \cdot \dot{J}_{\tau} \mu_{\tau}^{1, 0} + \beta_r \cdot
    \nabla_{g_{\tau}, \eta_{\tau}^{0, 1}}  \dot{J}_{\tau} \cdot \mu_{\tau}^{1,
    0}\\
    &  & \\
    & - & \nabla_{g_{\tau}, \mu_{\tau}^{1, 0}} \beta_r \cdot \dot{J}_{\tau}
    \eta_{\tau}^{0, 1} - \beta_r \cdot \nabla_{g_{\tau}, \mu_{\tau}^{1, 0}} 
    \dot{J}_{\tau} \cdot \eta_{\tau}^{0, 1} .
  \end{eqnarray*}
  We notice also the trivial identity $\beta_r \cdot \dot{J}_{\tau}
  \mu_{\tau}^{1, 0} = \beta_r \cdot ( \dot{J}_{\tau} \mu)_{_{J_{\tau}}}^{0, 1} = 0$, over $U$. Taking a covariant derivative
  of this we infer
  \begin{eqnarray*}
    0 & = & \nabla_{g_{\tau}, \eta_{\tau}^{0, 1}} \beta_r \cdot \dot{J}_{\tau}
    \mu_{\tau}^{1, 0} + \beta_r \cdot \nabla_{g_{\tau}, \eta_{\tau}^{0, 1}} 
    \dot{J}_{\tau} \cdot \mu_{\tau}^{1, 0} + \beta_r \cdot \dot{J}_{\tau}
    \nabla_{g_{\tau}, \eta_{\tau}^{0, 1}} \mu_{\tau}^{1, 0} .
  \end{eqnarray*}
  The identity (\ref{anti-Lin-CovJ}) implies
  \begin{eqnarray*}
    \nabla_{g_{\tau}, \eta_{\tau}^{0, 1}} \mu_{\tau}^{1, 0} & = & \left(
    \nabla_{g_{\tau}, \eta_{\tau}^{0, 1}} \mu - \frac{i}{2} \nabla_{g_{\tau},
    \eta} J \cdot \mu \right)_{_{J_{\tau}}}^{1, 0} .
  \end{eqnarray*}
  Thus
  \begin{eqnarray*}
    \beta_r \cdot \dot{J}_{\tau} \nabla_{g_{\tau}, \eta_{\tau}^{0, 1}}
    \mu_{\tau} & = & \beta_r \cdot \left[ \dot{J}_{\tau}  \left(
    \nabla_{g_{\tau}, \eta_{\tau}^{0, 1}} \mu - \frac{i}{2} \nabla_{g_{\tau},
    \eta} J \cdot \mu \right) \right]_{_{J_{\tau}}}^{0, 1} = \; 0,
  \end{eqnarray*}
  and
  \begin{equation}
    \label{dif-BJ} d \left( \beta_r \cdot \dot{J}_{\tau} \right) \left(
    \eta_{\tau}^{0, 1}, \mu_{\tau}^{1, 0} \right) \; = \; - \beta_r \cdot
    \nabla_{g_{\tau}, \mu_{\tau}^{1, 0}}  \dot{J}_{\tau} \cdot \eta_{\tau}^{0,
    1},
  \end{equation}
  at the point $p$, since
  \begin{eqnarray*}
    \nabla_{g_{\tau}, \mu_{\tau}^{1, 0}} \beta_r \cdot \dot{J}_{\tau}
    \eta_{\tau}^{0, 1} & = & \nabla_{g_{\tau}, \mu_{\tau}^{1, 0}} \beta_r
    \cdot \left( \dot{J}_{\tau} \eta \right)_{_{J_{\tau}}}^{1, 0} \; = \; 0,
  \end{eqnarray*}
  at $p$ thanks to (\ref{vanishLVCivita}). Taking a covariant derivative of
  the identity 
  $$
  \dot{J}_{\tau} J_{\tau} + J_{\tau}  \dot{J}_{\tau} = 0,
  $$
  we
  obtain
  \begin{eqnarray*}
    \nabla_{g_{\tau}}  \dot{J}_{\tau} J_{\tau} + \dot{J}_{\tau}
    \nabla_{g_{\tau}} J_{\tau} + \nabla_{g_{\tau}} J_{\tau}  \dot{J}_{\tau} +
    J_{\tau} \nabla_{g_{\tau}}  \dot{J}_{\tau}  =  0,
  \end{eqnarray*}
  and thus
  \begin{eqnarray*}
    2 \beta_r \cdot \nabla_{g_{\tau}, \mu_{\tau}^{1, 0}}  \dot{J}_{\tau} \cdot
    \eta^{0, 1}_{\tau} & = & 2 \beta_r \cdot \left( \nabla_{g_{\tau},
    \mu_{\tau}^{1, 0}}  \dot{J}_{\tau} \cdot \eta \right)_{_{J_{\tau}}}^{1, 0}\\
    & &\\
    &-& i \,\beta_r \cdot \left( \dot{J}_{\tau} \nabla_{g_{\tau}, \mu^{1,
    0}_{\tau}} J_{\tau} + \nabla_{g_{\tau}, \mu_{\tau}^{1, 0}} J_{\tau} 
    \dot{J}_{\tau} \right) \eta\\
    &  & \\
    & = & 2 \beta_r \cdot \nabla_{g_{\tau}, \mu_{\tau}^{1, 0}} 
    \dot{J}_{\tau} \cdot \eta - i \beta_r \cdot \dot{J}_{\tau}
    \nabla_{g_{\tau}, \mu} J_{\tau} \cdot \eta,
  \end{eqnarray*}
  thanks to (\ref{anti-Lin-CovJ}) and the fact that $\beta_r$ is of type
  $\left( 1, 0 \right)$ with respect to $J_{\tau}$. We deduce
  \begin{eqnarray*}
    - i d \left( \beta_r \cdot \dot{J}_{\tau} \right) \left( \eta_{\tau}^{0,
    1}, \mu_{\tau}^{1, 0} \right)  =  i\, \beta_r \cdot \left( \nabla^{1,
    0}_{g_{\tau}, J_{\tau}, \mu}  \dot{J}_{\tau} - \frac{1}{2} J_{\tau} 
    \dot{J}_{\tau} \nabla_{g_{\tau}, \mu} J_{\tau} \right) \eta,
  \end{eqnarray*}
  thanks to (\ref{dif-BJ}), and thus
  \begin{eqnarray}
    \label{var-dbar1form} \eta_{_{J_{\tau}}}^{0, 1} \neg 2 \frac{d}{d t}
    _{\mid_{t = \tau}} \left( \overline{\partial}_{_{J_t}} \beta^{1, 0}_{r, t}
    \right) &=& i\, \beta_r \cdot N_{_{J_{\tau}}} \left( \eta, \dot{J}_{\tau}
    \bullet \right) \nonumber
    \\\nonumber
    & &\\
    &+& i\, \beta_r \cdot \left( \nabla^{1, 0}_{g_{\tau},
    J_{\tau}}  \dot{J}_{\tau} - \frac{1}{2} J_{\tau}  \dot{J}_{\tau}
    \nabla_{g_{\tau}} J_{\tau} \right) \eta .
  \end{eqnarray}
  Using (\ref{vanishDbar}) and (\ref{var-dbar1form}) we obtain
  \begin{eqnarray*}
    2 \dot{\alpha}_{\tau} (\eta) & = & \dot{J}_{\tau} \eta . f_{\tau} \\
    & &\\
    &+& 2 \Re
    e \left\{ i \beta^{- 1}_{\tau} \sum_{l = 1}^n \beta_1 \wedge \ldots \wedge
    \left[ \eta_{_{J_{\tau}}}^{0, 1} \neg 2 \frac{d}{d t} _{\mid_{t = \tau}}
    \left( \overline{\partial}_{_{J_t}} \beta^{1, 0}_{l, t} \right) \right]
    \wedge \ldots \wedge \beta_n \right\}\\
    &  & \\
    & = & d f_{\tau} \cdot \dot{J}_{\tau} \eta \\
    & &\\
   & -& 2 \Re e
    \tmop{Tr}_{_{\mathbbm{C}}} \left[ N_{_{J_{\tau}}} \left( \eta,
    \dot{J}_{\tau} \bullet \right) + \left( \nabla^{1, 0}_{g_{\tau}, J_{\tau}}
    \dot{J}_{\tau} - \frac{1}{2} J_{\tau}  \dot{J}_{\tau} \nabla_{g_{\tau}}
    J_{\tau} \right) \eta \right]\\
    &  & \\
    & = & d f_{\tau} \cdot \dot{J}_{\tau} \eta \; - \;
    \tmop{Tr}_{_{\mathbbm{R}}} \left[ N_{_{J_{\tau}}} \left( \eta,
    \dot{J}_{\tau} \bullet \right) + \left( \nabla_{g_{\tau}}  \dot{J}_{\tau}
    - \frac{1}{2} J_{\tau}  \dot{J}_{\tau} \nabla_{g_{\tau}} J_{\tau} \right)
    \eta \right] .
  \end{eqnarray*}
  We show now the identity
  \begin{eqnarray}
    2 \tmop{Tr}_{_{\mathbbm{R}}} \left[ N_{_{J_{\tau}}} \left( \eta,
    \dot{J}_{\tau} \bullet \right) \right] & = & \tmop{Tr}_{_{\mathbbm{R}}}
    \left( J_{\tau}  \dot{J}_{\tau} \nabla_{g_{\tau}} J_{\tau} \cdot \eta
    \right) .\label{Trace1}
  \end{eqnarray}
  Indeed, let $\left( e_k \right)^{2 n}_{k = 1} \subset T_{X, p}$ be a
  $g_{\tau} \left( p \right)$-orthonormal basis. Using (\ref{covJ-Nij}),
  (\ref{anti-Lin-CovJ}) and the symmetry assumption $\dot{J}_{\tau} =
  ( \dot{J}_{\tau})_{g_{\tau}}^T$, we obtain
  \begin{eqnarray*}
    2 g \left( e_k, N_{_{J_{\tau}}} \left( \eta, \dot{J}_{\tau} e_k \right)
    \right) & = & g \left( \nabla_{g_{\tau}, J_{\tau} e_k} J_{\tau} \cdot
    \eta, \dot{J}_{\tau} e_k \right)\\
    &  & \\
    & = & - g \left( \dot{J}_{\tau} J_{\tau} \nabla_{g_{\tau}, e_k} J_{\tau}
    \cdot \eta, e_k \right)\\
    &  & \\
    & = & g \left( J_{\tau}  \dot{J}_{\tau} \nabla_{g_{\tau}, e_k} J_{\tau}
    \cdot \eta, e_k \right),
  \end{eqnarray*}
  and thus the required identity (\ref{Trace1}). We infer the formula
  \begin{eqnarray*}
    2 \dot{\alpha}_{\tau} (\eta) & = & -\,\tmop{Tr}_{_{\mathbbm{R}}} \left(
    \nabla_{g_{\tau}}  \dot{J}_{\tau} \cdot \eta \right) + d f_{\tau}
    \cdot \dot{J}_{\tau} \eta,
  \end{eqnarray*}
  over $U$. Using the symmetry identities $\dot{J_{\tau}} = ( \dot{J}_{\tau})_{g_{\tau}}^T$ and
  $\nabla_{g_{\tau}, \xi}  \dot{J}_{\tau} =( \nabla_{g_{\tau}, \xi} 
  \dot{J}_{\tau})_{g_{\tau}}^T$, we infer
  \begin{eqnarray*}
    2 \dot{\alpha}_{\tau} & = & \nabla_{g_{\tau}}^{\ast_{\Omega}} 
    \dot{J}_{\tau} \neg g_{\tau}\\
    &  & \\
    & = & -\,J_{\tau} \nabla_{g_{\tau}}^{\ast_{\Omega}} 
    \dot{J}_{\tau}  \neg \omega_{\tau},
  \end{eqnarray*}
  over $U$. We conclude the required variation formula for arbitrary time $t$,
  in the case of constant volume form. In the case of variable volume forms,
  we fix an arbitrary time $\tau$ and we time derive at $t = \tau$ the
  decomposition
  \begin{eqnarray*}
    \tmop{Ric}_{_{J_t}} (\Omega_t) & = & \tmop{Ric}_{_{J_t}} (\Omega_{\tau}) -
    d d_{_{J_t}}^c \log \frac{\Omega_t}{\Omega_{\tau}} .
  \end{eqnarray*}
  We obtain
  \begin{eqnarray*}
    \frac{d}{d t} _{\mid_{t = \tau}} \tmop{Ric}_{_{J_t}} (\Omega_t) & = &
    \frac{d}{d t} _{\mid_{t = \tau}} \tmop{Ric}_{_{J_t}} (\Omega_{\tau}) - d
    d_{_{J_{\tau}}}^c  \dot{\Omega}^{\ast}_{\tau},
  \end{eqnarray*}
  and thus
  \[ 2 \frac{d}{d t} \tmop{Ric}_{_{J_t}} (\Omega_t) = d \left[
     \left( J_t\nabla_{g_t}^{\ast_{\Omega_t}} \dot{J}_t
     - \nabla_{g_t}  \dot{\Omega}^{\ast}_t \right) \neg \, \omega_t 
     \right], \]
  thanks to the variation formula for the fixed volume form case. The
  conclusion follows from Cartan's identity for the Lie derivative of
  differential forms.
\end{proof}

We infer the following corollary.

\begin{corollary}
  \label{var-symp-Ricci}Let $\omega$ be a symplectic form and let $\left( J_t,
  \Omega_t \right)_t \subset \mathcal{J}^{\tmop{ac}}_{\omega} \times
  \mathcal{V}$ be an arbitrary smooth family. Then holds the variation formulas
  \begin{eqnarray}\label{CRvarG1}
    2 \frac{d}{d t} \tmop{Ric}_{J_t} (\Omega_t) & = & -
    L_{\nabla_{g_t}^{\ast_{\Omega_t}}  \dot{g}_t^{\ast} + \nabla_{g_t} 
    \dot{\Omega}_t^{\ast}} \omega - 2\, d \tmop{Tr}_{g_t}\left[\omega(\bullet
    \neg \,N_{_{J_t}})\dot{g}_t^{\ast}\right] ,
  \end{eqnarray}
  and
 \begin{eqnarray}\label{CRvarG2}
    2 \frac{d}{d t} \tmop{Ric}_{J_t} (\Omega_t) & = & -
    L_{\nabla_{g_t}^{\ast_{\Omega_t}}  \dot{g}_t^{\ast} + \nabla_{g_t} 
    \dot{\Omega}_t^{\ast}} \omega + d \tmop{Tr}_{g_t}\left[\omega(\bullet
    \neg \,\overline{\partial}_{T_{X, J_t}} \dot{g}_t^{\ast})\right] ,
  \end{eqnarray} 
with $g_t \assign - \omega J_t$.
\end{corollary}

\begin{proof}
  We have $\dot{g}^{\ast}_t = - J_t \dot{J_t}$ and thus the property
  $\dot{J}_t = ( \dot{J}_t)_{g_t}^T$, which allows to apply
  (\ref{Acvr-O-Rc-fm}). We notice now the equality 
 \begin{eqnarray*}
  \nabla_{g_t}^{\ast_{\Omega}} \left( J_t 
    \dot{J}_t \right)
    &=&-\nabla_{g_t,e_k} J_t\cdot \dot{J}_t e_k + J_t \nabla_{g_t}^{\ast_{\Omega}} 
    \dot{J}_t,
 \end{eqnarray*} 
with respect to a $g_t$-orthonormal local frame of $T_X$. We deduce
\begin{eqnarray*}
    2 \dot{\alpha}_t & = & -\,\left[\nabla_{g_t}^{\ast_{\Omega}} \left( J_t 
    \dot{J}_t \right)+\nabla_{g_t,e_k} J_t\cdot \dot{J}_t e_k\right] \neg \omega.
\end{eqnarray*}
The identity (\ref{covJ-Nij}) implies
\begin{eqnarray*}
 \omega(\nabla_{g_t,e_k} J_t\cdot \dot{J}_t e_k ,\xi) 
 & = &
 2 \,\omega(e_k, N_{_{J_t}}(\dot{J}_t e_k,J_t\xi))
 \\
 \\
 &=& 
 2 \,\omega(e_k, N_{_{J_t}}(J_t\dot{J}_t e_k,\xi))
 \\
 \\
 &=&
 -\,2\, d \tmop{Tr}_{g_t}\left[\omega(\xi
    \neg \,N_{_{J_t}})\dot{g}_t^{\ast}\right].
\end{eqnarray*}
We infer the variation formula (\ref{CRvarG1}). In order to show (\ref{CRvarG2}) we notice first the identity $\omega(e_k,N_{_{J_t}}(g^{-1}_t e^\ast_k,\xi))\equiv
0$, for arbitrary real local frame $(e_k)_k$ of $T_X$. Time deriving this we obtain
\begin{eqnarray*}
 \omega(e_k,N_{_{J_t}}(\dot{g}^{\ast}_t e_k,\xi))&=&\omega(e_k,\dot{N}_{_{J_t}}(e_k,\xi)) ,
\end{eqnarray*}
with respect to our $g_t$-orthonormal local frame. Then the
general formula
\begin{eqnarray*}
  2\,\frac{d}{d t} N_{J_t} & = & \overline{\partial}_{T_{X, J_t}}  (J_t\dot{J}_t) + J_t\dot{J}_t N_{J_t}-(J_t\dot{J}_t) \neg N_{J_t}  \,,
\end{eqnarray*}
(see the proof of lemma 7 in \cite{Pal3}), implies
\begin{eqnarray*}
 \omega(e_k,N_{_{J_t}}(\dot{g}^{\ast}_t e_k,\xi))&=&\omega(e_k,N_{_{J_t}}(\dot{g}^{\ast}_t e_k,\xi)+N_{_{J_t}}(e_k,\dot{g}^{\ast}_t\xi)-\dot{g}^{\ast}_tN_{_{J_t}}(e_k,\xi))
 \\
 \\
 &-&\omega(e_k,\overline{\partial}_{T_{X, J_t}}\dot{g}_t^{\ast}(e_k,\xi))
 \\
 \\
 &=&
 \omega(e_k,N_{_{J_t}}(\dot{g}^{\ast}_t e_k,\xi))-\omega(\dot{g}^{\ast}_t e_k,N_{_{J_t}}(e_k,\xi))\\
 \\
 &-&\omega(\overline{\partial}_{T_{X, J_t}}\dot{g}_t^{\ast}(\xi,e_k),e_k) .
\end{eqnarray*}
Assuming for simplicity that the $g_t$-orthonormal local frame diagonalizes $\dot{g}_t^{\ast}$ we deduce the identity
\begin{eqnarray*}
 2\, \tmop{Tr}_{g_t}\left[\omega(\bullet
    \neg \,N_{_{J_t}})\dot{g}_t^{\ast}\right]
    &=&
 -\tmop{Tr}_{g_t}\left[\omega(\bullet
    \neg \,\overline{\partial}_{T_{X, J_t}} \dot{g}_t^{\ast})\right] ,   
\end{eqnarray*}
which implies the variation formula (\ref{CRvarG2}).
\end{proof}

Combining lemma \ref{acx-geodesic}, proposition \ref{F-conserv} and corollary
\ref{var-symp-Ricci}, we deduce the main theorem \ref{Main}.

\section{The decomposition of the Bakry-Emery-Ricci tensor}

We compare first the Riemannian Ricci tensor $\tmop{Ric}(g)$ with the $\omega$-Chern-Ricci
tensor.
\begin{lemma}
  Let $\left( X, J, g \right)$ be an almost K\"ahler manifold with symplectic
  form $\omega \assign g J$. Then holds the identity
  \begin{equation}
    \label{Chern-Ricci-Rm-Ricci} \tmop{Ric}_{_J} (\omega) (\xi, J \eta) =
    \tmop{Ric} (g) \left( \xi, \eta \right) 
    +\omega(\nabla_g^\ast\widehat{\,\nabla_g J\,}\xi,\eta)
    + \frac{1}{4}
    \tmop{Tr}_{_{\mathbbm{R}}} \left( \nabla_{g_{} \nocomma \nocomma, \xi} J
    \cdot \nabla_{g_{} \nocomma \nocomma, \eta} J \right),
  \end{equation}
  where $\widehat{\,\nabla_g J\,}(\xi,\eta):=\nabla_g J(\eta,\xi)$.
\end{lemma}

\begin{proof}
  Using formula (\ref{Chern-LC}), the standard Curvature identity
  \begin{eqnarray*}
    \left( \nabla_{g_{} \nocomma \nocomma, \xi} \nabla_{g_{} \nocomma
    \nocomma, \eta} - \nabla_{g_{} \nocomma \nocomma, \eta} \nabla_{g_{}
    \nocomma \nocomma, \xi} - \nabla_{g_{} \nocomma \nocomma, [\xi, \eta]} 
    \right) \mu  = \mathcal{R}_g \left( \xi, \eta \right) \mu,
  \end{eqnarray*}
a similar one for the Chern curvature $\mathcal{C}_{\omega} \left( T_{X,
  J} \right)$ and the trivial equality
\begin{eqnarray}\label{CommutSec}
\nabla^2_{g,\xi,\eta} J- \nabla^2_{g,\eta,\xi} J=[\mathcal{R}_g \left( \xi, \eta \right) ,J],
\end{eqnarray}
we obtain the relation
  \begin{equation}
    \label{Chern-Rm} \mathcal{C}_{\omega} \left( T_{X, J} \right) \left( \xi,
    \eta \right) \mu =\mathcal{R}_g \left( \xi, \eta \right)^{1,0}_J \mu - \frac{1}{4}
    \left( \nabla_{g_{} \nocomma \nocomma, \xi} J \cdot \nabla_{g_{} \nocomma
    \nocomma, \eta} J - \nabla_{g_{} \nocomma \nocomma, \eta} J \cdot
    \nabla_{g_{} \nocomma \nocomma, \xi} J \right) \mu .
  \end{equation}
  Let now $\left( e_k \right)^n_{k = 1} \subset T_{X, p}$ be a $\omega \left(
  p \right)$-orthonormal and $J \left( p \right)$-complex basis. Then 
  \begin{eqnarray*}
    \tmop{Ric}_{_J} (\omega) \left( \xi, J \eta \right) & = & \sum_{k = 1}^n g
    \left( J\mathcal{C}_{\omega} \left( T_{X, J} \right) \left( \xi, J \eta
    \right) e_k, e_k \right) \\
    &  & \\
    & = & - \sum_{k = 1}^n g \left( \mathcal{C}_{\omega} \left( T_{X, J}
    \right) \left( \xi, J \eta \right) e_k, J e_k \right) ,
  \end{eqnarray*}
  thanks to the identity $[\mathcal{C}_{\omega} \left( T_{X, J} \right) \left(
    \xi, \eta \right), J] = 0$. 
Using formula (\ref{Chern-Rm}) we obtain
  \begin{eqnarray*}
    \tmop{Ric}_{_J} (\omega) \left( \xi, J \eta \right) & = & - \sum_{k = 1}^n
    g \left( \mathcal{R}_g \left( \xi, J \eta \right)^{1,0}_J e_k, J e_k \right)\\
    &  & \\
    & + & \frac{1}{4}  \sum_{k = 1}^n g \left( \left( \nabla_{g_{} \nocomma
    \nocomma, \xi} J \cdot \nabla_{g_{} \nocomma \nocomma, J \eta} J -
    \nabla_{g_{} \nocomma \nocomma, J \eta} J \cdot \nabla_{g_{} \nocomma
    \nocomma, \xi} J \right) e_k, J e_k \right) .
  \end{eqnarray*}
We notice now the equalities  
\begin{eqnarray*}
-g \left( \mathcal{R}_g \left( \xi, J \eta \right)^{1,0}_J e_k, J e_k \right)
&=&
-\frac{1}{2} g \left( \mathcal{R}_g \left( \xi, J \eta \right) e_k, J e_k \right)
\\
\\
&+&\frac{1}{2} g \left( \mathcal{R}_g \left( \xi, J \eta \right) J e_k, e_k \right)
\\
\\
&-& g \left( \mathcal{R}_g \left( \xi, J \eta \right) e_k, J e_k \right) ,
\end{eqnarray*}
thanks to the anti-symmetry identity $(\mathcal{R}_g \left( \xi, J \eta \right))^T_g=-\mathcal{R}_g \left( \xi, J \eta \right)$. 
Using the first Bianchi identity and the identity (\ref{anti-Lin-CovJ}), we infer
  \begin{eqnarray*}
    \tmop{Ric}_{_J} (\omega) \left( \xi, J \eta \right) & = & \sum_{k = 1}^n g
    \left( \mathcal{R}_g \left( J \eta, e_k \right) \xi +\mathcal{R}_g \left(
    e_k, \xi \right) J \eta, J e_k \right)\\
    &  & \\
    & + & \frac{1}{4}  \sum_{k = 1}^n g \left( \left( \nabla_{g_{} \nocomma
    \nocomma, \xi} J \cdot \nabla_{g_{} \nocomma \nocomma, \eta} J +
    \nabla_{g_{} \nocomma \nocomma, \eta} J \cdot \nabla_{g_{} \nocomma
    \nocomma, \xi} J \right) e_k, e_k \right) \\
    &  & \\
    & = & \sum_{k = 1}^n \left[ R_g \left( J \eta, e_k, J e_k, \xi \right) +
    g \left( \mathcal{R}_g \left( e_k, \xi \right) J\eta, J e_k \right) \right]\\
    &  & \\
    & + & \frac{1}{4} \tmop{Tr}_{_{\mathbbm{R}}} \left( \nabla_{g_{} \nocomma
    \nocomma, \xi} J \cdot \nabla_{g_{} \nocomma \nocomma, \eta} J \right),
  \end{eqnarray*}
  where $R_g \in C^{\infty} \left( X, S_{\mathbbm{R}}^2
  (\Lambda_{\mathbbm{R}}^2 T^{\ast}_X) \right)$ is the Riemannian curvature
  form. Using its symmetry properties we have
  \begin{eqnarray*}
    R_g \left( J \eta, e_k, J e_k, \xi \right) & = & R_g \left( J e_k, \xi, J
    \eta, e_k \right)\\
    &  & \\
    & = & - R_g \left( J e_k, \xi, e_k, J \eta \right)\\
    &  & \\
    & = & - g \left( \mathcal{R}_g \left( J e_k, \xi \right) J \eta, e_k
    \right),
  \end{eqnarray*}
We infer the equality
\begin{eqnarray*}
    \tmop{Ric}_{_J} (\omega) \left( \xi, J \eta \right) & = & -\sum_{k = 1}^n
    \left[ g \left( J\mathcal{R}_g \left( J e_k, \xi \right) J\eta, J e_k
    \right) + g \left( J\mathcal{R}_g \left( e_k, \xi \right) J\eta, e_k \right)
    \right]\\
    &  & \\
    & + & \frac{1}{4} \tmop{Tr}_{_{\mathbbm{R}}} \left( \nabla_{g_{} \nocomma
    \nocomma, \xi} J \cdot \nabla_{g_{} \nocomma \nocomma, \eta} J \right),
  \end{eqnarray*}  
Using (\ref{CommutSec}) we deduce
  \begin{eqnarray*}
    \tmop{Ric}_{_J} (\omega) \left( \xi, J \eta \right) & = & \sum_{k = 1}^n
    \left[ g \left( \mathcal{R}_g \left( J e_k, \xi \right) \eta, J e_k
    \right) + g \left( \mathcal{R}_g \left( e_k, \xi \right) \eta, e_k \right)
    \right]
 \\
    &  & \\
    & + &    
 \sum_{k = 1}^n
     g \left( J\left[\nabla^2_{g,\xi,J e_k} J-\nabla^2_{g,J e_k, \xi} J\right] \eta, J e_k
    \right)  
\\
    &  & \\
    & + &    \sum_{k = 1}^n  g \left( J \left[\nabla^2_{g,\xi,e_k}J  -\nabla^2_{g,e_k, \xi} J\right] \eta, e_k\right)
    \\
    &  & \\
    & + & \frac{1}{4} \tmop{Tr}_{_{\mathbbm{R}}} \left( \nabla_{g_{} \nocomma
    \nocomma, \xi} J \cdot \nabla_{g_{} \nocomma \nocomma, \eta} J \right),
  \end{eqnarray*}
  and thus
\begin{eqnarray}  
 \tmop{Ric}_{_J} (\omega) \left( \xi, J \eta \right) & = & \tmop{Ric}(g)\left( \xi,  \eta \right)\nonumber
  \\\nonumber
    &  & \\\nonumber
    & + &    
 \sum_{k = 1}^n
     g \left( \left[\nabla^2_{g,\xi,J e_k} J-\nabla^2_{g,J e_k, \xi} J\right] \eta, e_k
    \right)  \nonumber
\\\nonumber
    &  & \\
    & - &    \sum_{k = 1}^n  g \left( \left[\nabla^2_{g,\xi,e_k}J  -\nabla^2_{g,e_k, \xi} J\right] \eta, J e_k\right)\nonumber
    \\\nonumber
    &  & \\
    & + & \frac{1}{4} \tmop{Tr}_{_{\mathbbm{R}}} \left( \nabla_{g_{} \nocomma
    \nocomma, \xi} J \cdot \nabla_{g_{} \nocomma \nocomma, \eta} J \right).\label{interCRicci}
\end{eqnarray}    
We notice now that the identity $J_g^T = - J$, implies $\left( \nabla^2_{g,
\xi, \eta} J \right)_g^T = - \nabla^2_{g, \xi, \eta} J$. Using this we obtain
\begin{eqnarray*}
  \mathbbm{X}_k & \assign & g \left( \left[ \nabla^2_{g, \xi, J e_k}
  J_{_{_{_{}}}} - \nabla^2_{g, J e_k, \xi} J \right] \eta_{_{_{_{_{}}}}}, e_k
  \right)\\
  &  & \\
  & + & g \left( \left[ \nabla^2_{g, e_k, \xi} J_{_{_{_{}}}} - \nabla^2_{g,
  \xi, e_k} J \right] \eta, J_{_{_{_{_{}}}}} e_k \right)\\
  &  & \\
  & = & g \left( \eta_{_{_{_{_{}}}}}, \left[ \nabla^2_{g, J e_k, \xi} J -
  \nabla^2_{g, \xi, J e_k} J_{_{_{_{}}}} \right] e_k + \left[ \nabla^2_{g,
  \xi, e_k} J - \nabla^2_{g, e_k, \xi} J_{_{_{_{}}}} \right] J e_k \right) \\
  &  & \\
  & = & g \left( J \eta_{_{_{_{_{}}}}}, J \left[ \nabla^2_{g, J e_k, \xi} J -
  \nabla^2_{g, \xi, J e_k} J_{_{_{_{}}}} \right] e_k + J \left[ \nabla^2_{g,
  \xi, e_k} J - \nabla^2_{g, e_k, \xi} J_{_{_{_{}}}} \right] J e_k \right) .
\end{eqnarray*}
Taking a covariant derivative of the identity $\nabla_g J \cdot J = - J
\nabla_g J$ we infer
\begin{eqnarray*}
  J \nabla^2_{g, \xi, \eta} J & = & - \nabla^2_{g, \xi, \eta} J \cdot J -
  \nabla_{g, \xi} J \cdot \nabla_{g, \eta} J - \nabla_{g, \eta} J \cdot
  \nabla_{g, \xi} J .
\end{eqnarray*}
Using this we deduce
\begin{eqnarray*}
  &  & J \left[ \nabla^2_{g, J e_k, \xi} J - \nabla^2_{g, \xi, J e_k}
  J_{_{_{_{}}}} \right] e_k + J \left[ \nabla^2_{g, \xi, e_k} J - \nabla^2_{g,
  e_k, \xi} J_{_{_{_{}}}} \right] J e_k\\
  &  & \\
  & = & \left[ \nabla^2_{g, \xi, J e_k} J_{_{_{_{}}}} - \nabla^2_{g, J e_k,
  \xi} J \right] J e_k + \left[ \nabla^2_{g, \xi, e_k} J - \nabla^2_{g, e_k,
  \xi} J_{_{_{_{}}}} \right] e_k\\
  &  & \\
  & - & \nabla_{g, J e_k} J \cdot \nabla_{g, \xi} J e_k - \nabla_{g, \xi} J
  \cdot \nabla_{g, J e_k} J e_k\\
  &  & \\
  & + & \nabla_{g, \xi} J \cdot \nabla_{g, J e_k} J e_k + \nabla_{g, J e_k} J
  \cdot \nabla_{g, \xi} J e_k\\
  &  & \\
  & - & \nabla_{g, \xi} J \cdot \nabla_{g, e_k} J \cdot J e_k - \nabla_{g,
  e_k} J \cdot \nabla_{g, \xi} J \cdot J e_k\\
  &  & \\
  & + & \nabla_{g, e_k} J \cdot \nabla_{g, \xi} J \cdot J e_k + \nabla_{g,
  \xi} J \cdot \nabla_{g, e_k} J \cdot J e_k\\
  &  & \\
  & = & \left[ \nabla^2_{g, \xi, J e_k} J_{_{_{_{}}}} - \nabla^2_{g, J e_k,
  \xi} J \right] J e_k + \left[ \nabla^2_{g, \xi, e_k} J - \nabla^2_{g, e_k,
  \xi} J_{_{_{_{}}}} \right] e_k,
\end{eqnarray*}
by obvious diagonal cancellations. We notice now that the identity (4.4)
implies $\nabla^{\ast}_g J = 0$. A covariant derivative of this identity
implies $\tmop{Tr}_g \left( \xi \neg \nabla_g^2 J \right) = 0$. We deduce
\begin{eqnarray*}
  \sum_{k = 1}^n \mathbbm{X}_k & = & g \left( \nabla^{\ast}_g
  \widehat{\nabla_g J} \xi, J \eta \right) .
\end{eqnarray*}
This combined with (\ref{interCRicci}) implies the required formula (\ref{Chern-Ricci-Rm-Ricci}). 
\end{proof}

Let $\Omega > 0$ be a smooth volume form over an oriented Riemannian manifold
$(X, g)$. We define the $\Omega$-Bakry-Emery-Ricci tensor of $g$ as
\begin{eqnarray*}
  \tmop{Ric}_g (\Omega)  \assign  \tmop{Ric} (g)  + \nabla_g d \log \frac{dV_g}{\Omega} .
\end{eqnarray*}
We set for notations simplicity
\begin{eqnarray*}
  \left[ \tmop{Tr}_{_{\mathbbm{R}}} \left( \nabla_{g_{} \nocomma \nocomma,
  \bullet} J \cdot \nabla_{g_{} \nocomma \nocomma, \bullet} J \right) \right]
  \left( \xi, \eta \right)  \assign  \tmop{Tr}_{_{\mathbbm{R}}} \left(
  \nabla_{g_{} \nocomma \nocomma, \xi} J \cdot \nabla_{g_{} \nocomma \nocomma,
  \eta} J \right) .
\end{eqnarray*}
\begin{lemma}
  Let $\left( X, J, g \right)$ be an almost K\"ahler manifold with symplectic
  form $\omega \assign g J$ and let $\Omega > 0$ be a smooth volume form. Then
  hold the decomposition formula
  \begin{eqnarray}
    \label{cx-dec-Ric} \tmop{Ric}_g (\Omega) &=& - \tmop{Ric}_{_J} (\Omega) J 
    +\omega(\bullet,\nabla_g^\ast\widehat{\,\nabla_g J\,}\bullet)
     - \frac{1}{4}
    \tmop{Tr}_{_{\mathbbm{R}}} \left( \nabla_{g_{} \nocomma \nocomma, \bullet}
    J \cdot \nabla_{g_{} \nocomma \nocomma, \bullet} J \right)\nonumber
    \\\nonumber
    & &\\
    &+&g \left( \overline{\partial}_{T_{X, J}} \nabla_g \log \frac{dV_g}{\Omega}
    - \nabla_g \log \frac{dV_g}{\Omega} \neg N_{_J} \right) .
  \end{eqnarray}
\end{lemma}

\begin{proof}.
  This formula follows directly from the identities 
  $$
  \tmop{Ric}_{_J} (\Omega)=\tmop{Ric}_{_J} (\omega^n)+d d^c_{_J}\log \frac{\,\omega^n}{\Omega}\,,
  $$
 $\tmop{Ric}_{_J} (\omega^n)=\tmop{Ric}_{_J} (\omega)$, the identity (\ref{Chern-Ricci-Rm-Ricci}) and the decomposition formula
  \begin{eqnarray*}
    \nabla_g df = - \hspace{0.25em} d d^c_{_J} f \cdot J +
    g \left( \overline{\partial}_{T_{X, J}} \nabla_g f - \nabla_g f \neg
    N_{_J} \right),
\end{eqnarray*}  
  for all twice differentiable function $f$. The latter follows from a
  straightforward modification of the proof of lemma 29 in \cite{Pal4}.
\end{proof}
\\
\\
{\noindent}\tmtextbf{Acknowledgments.} I warmly thank the referee for pointing out a few inaccuracies in the original version of this manuscript.

\vspace{0.6cm}
\noindent
Nefton Pali
\\
Universit\'{e} Paris Sud, D\'epartement de Math\'ematiques 
\\
B\^{a}timent 425 F91405 Orsay, France
\\
E-mail: \textit{nefton.pali@math.u-psud.fr}
\end{document}